\documentclass[a4,12pt,leqno]{amsart}
\usepackage{amsmath,amssymb,latexsym,cite,mathtools}
\usepackage[colorlinks=true,urlcolor=blue,citecolor=red,linkcolor=blue,linktocpage,pdfpagelabels,bookmarksnumbered,bookmarksopen]{hyperref}
\newtheorem{teo}{Theorem}[section]
\newtheorem{lem}[teo]{Lemma}

\newtheorem{prop}[teo]{Proposition}
\newtheorem{oss}[teo]{Remark}
\theoremstyle{definition}
\newtheorem{defi}[teo]{Definition}
\makeatletter
\def\cleardoublepage{\clearpage\if@twoside \ifodd\c@page\else
\hbox{}
\thispagestyle{empty}
\newpage
\if@twocolumn\hbox{}\newpage\fi\fi\fi}
\makeatother

\title{H-convergence for equations depending on monotone operators in Carnot groups}

\author{A. Maione}

\address{Alberto Maione: Dipartimento di Matematica\\Universit\`a di Trento\\ Via Sommarive 14\\ 38123, Povo (Trento) - Italy\\}

\email{alberto.maione@unitn.it}

\thanks{A.M is supported by MIUR, Italy, GNAMPA of INDAM and University of Trento, Italy.}

\date{\today}

\begin{document}

%%%%%%%%%%%%%%%%%%%%%%%%%%%%%%%%%%%%%%%%%%%%

\begin{abstract}
The aim of this paper is to present some results, related to the convergence of solutions of Dirichlet problems for sequences of monotone operators, that generalize well-known results of Murat-Tartar, De Arcangelis-Serra Cassano and Baldi-Franchi-Tchou-Tesi to the sub-Riemannian framework of Carnot groups.
\end{abstract}
\maketitle

%%%%%%%%%%%%%%%%%%%%%%%%%%%%%%%%%%%%%%%%%%%%

\section{Introduction}
The \textit{$H$-convergence} was coined by Fran\c{c}ois Murat and Luc Tartar in the 70's and it addresses differential operators. In \cite{T,T2}, Tartar reports applications of the $H$-convergence to many different frameworks covering, among other things, the case involving monotone operators (see Definition \ref{def}) of the form
\begin{align*}
\mathcal{A}(u):=-\mathrm{div}(A(x,\nabla u)),
\end{align*}
where $A$ is a Carath\'eodory function satisfying uniformly ellipticity and continuous conditions, in the setting of Hilbert spaces \cite[Chapter 11]{T}.

In the last few years, this theory was developed by many different authors, as well as its applications to homogenization problems (see e.g. \cite{BMT,BMT2,BPT1,BPT2,BT,BCPDF,CPDF,CPDMDF,DASC2,DF, DGS,FMT,M,MT,SC}). In particular, in \cite{DASC}, De Arcangelis and Serra Cassano provided an extension of the original Murat and Tartar's $H$-compactness result related to monotone operators \cite[Theorem 11.2]{T}, working with a class of operators depending on weights, in the general setting of Banach spaces.
\medskip

% Roughly speaking, the Div-curl Lemma states that the scalar product of two sequences of vector fields, weakly convergent in $L^2(\mathbb{R}^n)$ and such that the sequence of the \textit{divergence} of the first sequence and the sequence of the \textit{curl} of the second one are compact in $H_{loc}^{-1}(\mathbb{R}^n)$ and $H_{loc}^{-1}(\mathbb{R}^n)^\frac{n(n-1)}{2}$ respectively, converges (strongly) in the sense of distributions to the scalar product of the limits of the starting sequences.

As for many other related topics in literature, it would be interesting to investigate for this kind of results in non--Euclidean settings. In this paper, we look for extensions to Carnot groups, that are connected, simply connected and nilpotent Lie groups, whose Lie algebra is stratified (see Section \ref{sect2} for details). These spaces are an example of
sub-Riemannian manifolds and they have become an environment of particular interest for analysis and PDEs over the last decades, see e.g. \cite{CMSV,MBP}.

A linear counterpart of the study of the $H$-convergence for monotone operators in Carnot groups, was faced up by Baldi, Franchi, Tchou and Tesi in \cite{BFT,FTT,BFTT}. The willing of adapting the original techniques in this new framework, needed a generalization of the Murat and Tartar's \textit{Div-curl Lemma} \cite[Lemma 7.2]{T}, which is a classic key tool in this theory. Therefore, after defining a suitable notion of \textit{intrinsic curl} for Carnot groups, and stating a version of the Div-curl Lemma compatible with this setting \cite[Theorem 5.1]{BFTT}, the authors were able to provide an extension of the $H$-compactness theorem for the linear case of \textit{matrix-valued measurable functions}, i.e., studying, in Hilbert spaces, operators of the form
$$\mathcal{A}(u):=-\mathrm{div}_\mathbb{G}(A(x)\nabla_\mathbb{G}u),$$
where $A$ is a ($m\times m$)-matrix-valued measurable function for $m\leq n$, dimension of the first layer of the Lie algebra associated with $\mathbb{G}$ (see Definition  \ref{Carnot} below), and $\nabla_{\mathbb{G}}$ and $\mathrm{div}_\mathbb{G}$ are the intrinsic gradient and the intrinsic divergence, respectively (see Definition \ref{def_curl_nabla} for details).
\medskip

Motivated by the previous results, in this paper we provide a new $H$-compactness theorem for monotone operators in Carnot groups, in the nonlinear case and in the general setting of Banach spaces, that is, working with operators of the form
\begin{align}\label{operator}
\mathcal{A}(u):=-\mathrm{div}_{\mathbb{G}}(A(\cdot,\nabla_{\mathbb{G}}u))\ \text{in}\ \Omega\subset\mathbb{G}\ \text{open}
\end{align}
for a given $A\in \mathcal{M}(\alpha,\beta;\Omega)$. The class $\mathcal{M}(\alpha,\beta;\Omega)$ is defined as follows
\begin{defi}\label{M,alpha,beta}
Let $\Omega\subset\mathbb{G}$ be open, $2\leq p<\infty$ and let $\alpha\leq\beta$ be positive constants. We define $\mathcal{M}(\alpha,\beta;\Omega)$ the class of Carath\'eodory functions $A: \Omega\times \mathbb{R}^{m}\to \mathbb{R}^{m}$ such that
\begin{itemize}
	\item [$(i)$] $A(x,0)=0$;
	\item [$(ii)$] $\left\langle A(x,\xi)-A(x,\eta), \xi-\eta\right\rangle\geq \alpha |\xi-\eta|^p$;
	\item [$(iii)$] $|A(x,\xi)-A(x,\eta)|\leq\beta\left[1+|\xi|^p+|\eta|^p\right]^\frac{p-2}{p}|\xi-\eta|$
\end{itemize}
for every $\xi,\eta\in \mathbb{R}^{m}$ and a.e. $x\in\Omega$.
\end{defi}
The structure of the paper is the following one: in Section \ref{sect2}, are stated the definitions of Carnot groups and the functional setting, needed throughout the paper.
In Section \ref{sect3}, we study the main properties of the class of monotone operators we are interested in, and finally, in Section \ref{sect4}, after defining a proper notion of $H$-convergence (see Definition \ref{Hconv} below), we provide the main result of the paper, namely
\begin{teo}\label{Xth}
Let $\Omega\subset\mathbb{G}$ be open, connected and bounded, $2\leq p<\infty$, $\alpha\leq\beta$ positive constants and let $(A^n)_n\subset\mathcal{M}(\alpha,\beta;\Omega)$. Then, up to subsequences, there exists $A^{eff}\in\mathcal{M}(\alpha,\beta;\Omega)$ such that
$$A^n\ \text{$H$-converges to}\ A^{eff}.$$
\end{teo}
We would like to stress that Theorem \ref{Xth} generalizes:
\begin{itemize}
    \item \cite[Theorem 11.2]{T}, for $\mathbb{G}=\mathbb{R}^n$ and $p=2$;
    \item \cite[Theorem 4.4]{FTT}, for $\mathbb{G}=\mathbb{H}$, $p=2$ and $\mathcal{A}(u)=-\mathrm{div}_\mathbb{H}(A(x)\nabla_\mathbb{H}u)$;
    \item \cite[Theorem 6.4]{BFT}, for $\mathbb{G}=\mathbb{H}^n$, $p=2$ and $\mathcal{A}(u)=-\mathrm{div}_\mathbb{H}(A(x)\nabla_\mathbb{H}u)$;
    \item \cite[Theorem 5.4]{BFTT}, for $p=2$ and $\mathcal{A}(u)=-\mathrm{div}_\mathbb{G}(A(x)\nabla_\mathbb{G}u)$.
\end{itemize}

%%%%%%%%%%%%%%%%%%%%%%%%%%%%%%%%%%%%%%%%%%%

\section{Preliminaries}\label{sect2}
\subsection{Carnot groups}
Let us recall just few definitions concerning Carnot groups. For a deeper study in this topic, the interested reader is referred to \cite{BLU}.
\begin{defi}\label{Carnot}
A Carnot group $\mathbb{G}$ of step $k$ is a connected, simply connected and nilpotent Lie group, whose Lie algebra $\mathfrak{g}$ admits a step $k$ stratification, that is, there exist $V_1,..,V_k$ linear subspaces of $\mathfrak{g}$, usually called layers, such that
	\begin{itemize}
		\item[$(i)$] $\mathfrak{g}=V_1\oplus..\oplus V_k$;
		\item[$(ii)$] $[V_1,V_i]=V_{i+1}$ for any $i<k$, where $[V_1,V_i]$ is the sub--algebra of $\mathfrak{g}$ generated by the commutation $[X,Y]$ of $X\in V_1,Y\in V_i$;
		\item[$(iii)$] $V_k\neq\{0\}$ and $V_i=\{0\}$ for all $i>k$, where $0$ is the identity element of $\mathfrak{g}$.
	\end{itemize}
\end{defi}
The simplest examples of Carnot groups are the Euclidean spaces, that are the only \textit{abelian} Carnot groups. Other typical examples of Carnot groups are the Heisenberg groups and the upper triangular groups.

It is clear from Definition \ref{Carnot}, that the first layer $V_1$ plays the role of generator of the algebra $\mathfrak{g}$, by commutation. For this reason, we will call $V_1$ the \textit{horizontal layer} and the other layers $V_i$, for $1<i\leq k$, the \textit{vertical layers}.
\medskip

We can define two different dimensions on $\mathbb{G}$: the \textit{topological dimension}, which is its dimension as Lie group, i.e.,
\begin{align*}
	{\rm dim}(\mathbb{G})={\rm dim}(\mathfrak{g})=\sum_{i=1}^km_i
\end{align*}
where $m_i:={\rm dim}(V_i)$ for any $i$, and the \textit{homogeneous dimension}, defined as
	$$Q:=\sum_{i=1}^ki\ m_i.$$
Let us notice that, when $\mathbb{G}$ is not $\mathbb{R}^n$, the homogeneous dimension of $\mathbb{G}$ is always bigger than the topological one.

From now on, we denote the dimension of the horizontal layer by $m$, instead of $m_1$.

\subsection{The functional setting}
Through the paper, $(X_1,..,X_m)$ denotes a basis of the horizontal layer $V_1$, $|\Omega|$ the Lebesgue measure of any set $\Omega\subset\mathbb{G}$ and, if $\xi,\eta\in\mathbb{R}^{m}$, we denote by $|\xi|$ and $\langle\xi,\eta\rangle$ the Euclidean norm and the scalar product, respectively.
The subbundle of the tangent bundle $T\mathbb{G}$, which is spanned by the vector fields $X_1,..,X_m$, is called the {\it horizontal bundle} and it is denoted by $H\mathbb{G}$. The sections of $H\mathbb{G}$ are called {\it horizontal sections}.
\begin{defi}\label{def_curl_nabla}
Let $(X_1,..,X_{m})$ be a basis of $V_1$, let $u:\mathbb{G}\rightarrow\mathbb{R}$ for which the partial derivatives $X_iu$ exist and let $\Phi$ a horizontal section such that $X_i\Phi_i\in L_{\rm loc}^1(\mathbb{G})$ for $i=1,..,m$. The \textit{intrinsic gradient} of $u$ and the \textit{intrinsic divergence} of $\Phi$ are defined, respectively, as
\begin{itemize}
    \item[$\cdot$] $\nabla_\mathbb{G}u:=\sum_{j=1}^m(X_j u)X_j=(X_1u,..,X_{m}u)$;
    \item[$\cdot$] ${\rm div}_\mathbb{G}(\Phi):=\sum_{i=1}^{m}X_i\Phi_i$.
\end{itemize}
\end{defi}
A definition of \textit{intrinsic curl} in the setting of Carnot groups, ${\rm curl}_\mathbb{G}$, can be found in \cite[Section 5]{BFTT}.
\begin{defi}\label{Definition 1.1.1} For any $1\leq p<\infty$ we set
\begin{itemize}
	\item[$\cdot$] $W^{1,p}_{\mathbb{G}}(\Omega):=\left\{u\in L^p(\Omega): X_j u\in L^ p(\Omega)\ {\hbox{\rm for }} j=1,\dots,m\right\}$\\
	endowed with its natural norm;
	\item[$\cdot$] $W^{1,p}_{\mathbb{G},0}(\Omega)$ will denote the closure of $C^\infty_c(\Omega)\cap W^{1,p}_{\mathbb{G}}(\Omega)$ in $W^{1,p}_{\mathbb{G}}(\Omega)$. If $\Omega$ is bounded, we can take $$\|u\|_{W^{1,p}_{\mathbb{G},0}(\Omega)}^p:=\int_{\Omega}|\nabla_{\mathbb{G}}u|^p\, dx$$
	for every $u\in W^{1,p}_{\mathbb{G},0}(\Omega)$, by the Poincar\'e inequality (see for instance \cite{MPSC2} and, for more details, \cite{MPSC,MV});
	\item[$\cdot$] $W^{-1,p'}_{\mathbb{G}}(\Omega):=(W^{1,p}_{\mathbb{G},0}(\Omega))'$;
	\item[$\cdot$] $L^p({\Omega,H\mathbb{G}})$ will denote the set of all measurable sections $\Phi\in L^p(\Omega)^m$.
\end{itemize}
\end{defi}
\begin{prop}{\cite[Corollary 4.14]{F}}
If $1<p<\infty$, then the space $W^{1,p}_\mathbb{G}(\Omega)$ is independent of the choice of $(X_1,..,X_{m})$.
\end{prop}

\subsection{Monotone operators}
\begin{defi}\label{def}
Let $V$ be a reflexive Banach space, $V^*$ its dual space and let $\mathcal{A}:V\to V^{*}$ be a mapping. We say that 
\begin{itemize}
\item[$\cdot$] $\mathcal{A}$ is \textit{monotone}, if
\[\left\langle \mathcal{A}(u)-\mathcal{A}(v), u-v \right\rangle_{V^{*}\times V}\geq 0\quad\mbox{for all}\ u,v\in V;\]
\item[$\cdot$] $\mathcal{A}$ is \textit{strictly-monotone}, if it is monotone and
\[\left\langle \mathcal{A}(u)-\mathcal{A}(v), u-v \right\rangle_{V^{*}\times V}= 0\quad \mbox{implies}\ u=v;\]
\item[$\cdot$] $\mathcal{A}$ is \textit{coercive}, if there exists an element $v \in V$ such that
\[\frac{\left\langle \mathcal{A}(u)-\mathcal{A}(v), u-v \right\rangle_{V^{*}\times V}}{\|u-v\|_V}\to \infty \quad\mbox{as}\ \|u\|_V\to \infty\ \mbox{for any}\ u\in V;\]
\item[$\cdot$] $\mathcal{A}$ is \textit{continuous on finite dimensional subspaces} of $V$ if, for any finite dimensional subspace $M\subset V$, the restriction of $\mathcal{A}$ to $M$ is weakly continuous, namely, if $\mathcal{A}: M\to V^{*}$ is weakly continuous.
\end{itemize}
\end{defi}
\begin{oss}
Let us remark that the operators defined in \eqref{operator}, are strictly-monotone, by Definition \ref{M,alpha,beta}. 
\end{oss}
The following result will be crucial later on.
\begin{teo}\cite[Corollary 1.8, Chapter III]{KS}\label{thKS}
Let X be a Banach space, let $K$ be a closed, nonempty and convex subset of $X$ and let $A:K\to X^{*}$ be monotone, coercive and continuous on finite dimensional subspaces of $K$. Then, there exists $u\in K$ such that
\[
\left\langle A(u), v-u\right\rangle_{X^*\times K}\geq 0\quad \mbox{for any}\ v\in K.
\]
\end{teo}

%%%%%%%%%%%%%%%%%%%%%%%%%%%%%%%%%%%%%%%%%%%

\section{Existence results for equations driven by monotone operators}\label{sect3}
From now on, let $\Omega\subset\mathbb{G}$ be open, connected and bounded, let $2\leq p<\infty$, $V=W^{1,p}_{\mathbb{G},0}(\Omega)$ and let $V^{*}=W^{-1,p'}_{\mathbb{G}}(\Omega)$. Moreover, let $\mathcal{A}:V\to V^{*}$ be defined as in \eqref{operator}.

As a corollary of the following result, $\mathcal{A}$ is continuous and invertible in $V$.
\begin{prop}\label{existence}
Let $A\in\mathcal{M}(\alpha,\beta;\Omega)$. Then, for every $f\in V^{*}$ there exists a unique (weak) solution $u\in V$ of 
\begin{equation}\label{1}
-\mathrm{div}_{\mathbb{G}}(A(\cdot,\nabla_{\mathbb{G}}u))=f\qquad \mbox{in}\ \Omega,
\end{equation}
i.e.,
\begin{equation}\label{111}
\int_{\Omega}\langle A(x,\nabla_{\mathbb{G}}u),\nabla_{\mathbb{G}} \varphi\rangle\, dx=\int_{\Omega} f\varphi\, dx\quad\forall\varphi\in C^{\infty}_c(\Omega).
\end{equation}
\end{prop}
\begin{oss}\label{approx}
By standard approximation arguments, (\ref{111}) holds for every $\varphi\in V$.
\end{oss}
\begin{proof}
Let $f\in V^*$ and let $\Phi:V\to V^{*}$ be defined as
\begin{align*}
\langle\Phi(u),v\rangle_{V^*\times V}:=\int_{\Omega}\left[\langle A(x,\nabla_{\mathbb{G}}u),\nabla_{\mathbb{G}} v\rangle-fv\right]\, dx\quad\forall u,v\in V.
\end{align*}
Let us show that $\Phi$ is strictly-monotone, coercive and continuous on finite dimensional subspaces of $V$.
\begin{oss}\label{weak-strong}
Since the weak convergence in finite dimensional Banach spaces is equivalent to the strong one, it is enough to prove the strongly continuity of $\Phi$, in the whole space $V$.
\end{oss}
For any $u,v\in V$, we have
\begin{align*}
\left\langle \Phi(u)-\Phi(v), u-v\right\rangle_{V^{*}\times V}&=\int_{\Omega}\langle A(x,\nabla_{\mathbb{G}} u)-A(x,\nabla_{\mathbb{G}}v),\nabla_{\mathbb{G}} (u-v)\rangle\, dx\\
&\geq \alpha\int_{\Omega} |\nabla_{\mathbb{G}} u-\nabla_{\mathbb{G}} v|^p\, dx=\alpha\|u-v\|_V^p\geq 0
\end{align*}
and
\begin{align*}
\frac{\left\langle \Phi(u)-\Phi(v), u-v \right\rangle_{V^{*}\times V}}{\|u-v\|_V}\geq\alpha\|u-v\|_V^{p-1}.
\end{align*}
Let now $(u_n)_n\subset V$ and $u\in V$ be such that $u_n\to u$ in $V$. By the H\"older inequality, we have
\begin{align*}
\left\langle \Phi(u_n)-\Phi(u), u_n-u\right\rangle_{V^{*}\times V}&=\int_{\Omega}\langle A(x,\nabla_{\mathbb{G}} u_n)-A(x,\nabla_{\mathbb{G}}u),\nabla_{\mathbb{G}}(u_n-u)\rangle\, dx\\
&\leq \|A(\cdot,\nabla_{\mathbb{G}} u_n)-A(\cdot,\nabla_{\mathbb{G}}u)\|_{L^{p'}(\Omega,H\mathbb{G})}\|u_n-u\|_V.
\end{align*}
Moreover, $A(\cdot,\nabla_{\mathbb{G}} u_n)\to A(\cdot,\nabla_{\mathbb{G}}u)$ in $L^{p'}(\Omega,H\mathbb{G})$, since
\begin{align*}
\|A(\cdot,&\nabla_{\mathbb{G}} u_n)-A(\cdot,\nabla_{\mathbb{G}}u)\|_{L^{p'}(\Omega,H\mathbb{G})}^{p'}=\int_{\Omega}|A(x,\nabla_{\mathbb{G}} u_n)-A(x,\nabla_{\mathbb{G}}u)|^{p'}\, dx\\&\leq\beta^{p'}\int_{\Omega}\left[1+|\nabla_{\mathbb{G}}u_n|^p+|\nabla_{\mathbb{G}}u|^p\right]^\frac{p-2}{p-1}|\nabla_{\mathbb{G}} u_n-\nabla_{\mathbb{G}} u|^{p'}\, dx\\&\leq\beta^{p'}\left(\int_{\Omega}\left[1+|\nabla_{\mathbb{G}}u_n|^p+|\nabla_{\mathbb{G}}u|^p\right]\, dx\right)^\frac{p-2}{p-1}\left(\int_{\Omega}|\nabla_{\mathbb{G}} u_n-\nabla_{\mathbb{G}} u|^{p}\, dx\right)^\frac{p'}{p}\\&=\beta^{p'}\left[|\Omega|+\|u_n\|_V^p+\|u\|_V^p\right]^{\frac{p-2}{p}p'}\|u_n-u\|_V^{p'}.
\end{align*}
%i.e.,
%\begin{align*}
%\|A(\cdot,\nabla_{\mathbb{G}} u_n)-A(\cdot,\nabla_{\mathbb{G}}u)\|_{L^{p'}(\Omega,H\mathbb{G})}\leq\beta\left[|\Omega|+\|u_n\|_V^p+\|u\|_V^p\right]^\frac{p-2}{p}\|u_n-u\|_V.
%\end{align*}

Therefore, by Theorem \ref{thKS}, there exists $u\in V$ such that 
\begin{align*}
\langle\Phi(u),v-u\rangle_{V^*\times V}\geq 0\quad \forall v\in V.
\end{align*}
Thus, for $v_1:=u+\varphi$ and $v_2:=u-\varphi$, we get
\begin{align*}
\langle\Phi(u),\varphi\rangle_{V^*\times V}=0\quad \forall\varphi\in V.
\end{align*}

Let now $u,v\in V$ be weak solutions of (\ref{1}). By Remark \ref{approx}
\begin{align*}
\int_{\Omega}\langle A(x,\nabla_{\mathbb{G}} u)-A(x,\nabla_{\mathbb{G}}v),\nabla_{\mathbb{G}}\varphi\rangle\, dx=0\quad\forall\varphi\in V.
\end{align*}
Therefore, for $\varphi=u-v$, we get the uniqueness of the solutions, since
\begin{align*}
0&=\int_{\Omega}\langle A(x,\nabla_{\mathbb{G}} u)-A(x,\nabla_{\mathbb{G}}v),\nabla_{\mathbb{G}}u-\nabla_{\mathbb{G}}v\rangle\, dx\\&\geq \alpha\int_{\Omega}|\nabla_{\mathbb{G}}(u-v)|^p\, dx=\alpha\|u-v\|_V^p\geq 0.
\end{align*}
\end{proof}
\begin{prop}\label{estimates}
Let $A\in\mathcal{M}(\alpha,\beta;\Omega)$, let $\mathcal{A}$ be defined as in $(\ref{operator})$ and let $\mathcal{A}^{-1}$ be its inverse operator. Then, the following three estimates hold:
\begin{itemize}
	\item [($a$)] $\langle\mathcal{A}(u)-\mathcal{A}(v),u-v\rangle_{V^*\times V}\geq\alpha\|u-v\|_V^p$;
	\item [($b$)] $\|\mathcal{A}^{-1}(f)-\mathcal{A}^{-1}(g)\|_V^p\leq\left(\frac{1}{\alpha}\right)^{p'}\|f-g\|_{V^*}^{p'}$;
	\item [($c$)] $\|\mathcal{A}(u)-\mathcal{A}(v)\|_{V^*}\leq\beta\left[|\Omega|+\|u\|_V^p+\|v\|_V^p\right]^{\frac{p-2}{p}}\|u-v\|_V$,
\end{itemize}
for all $u,v\in V$, for all $f,g\in V^*$.
\end{prop}
\begin{proof}
Let $u,v\in V$ and let $f,g\in V^*$ be such that $\mathcal{A}(u)=f$ and $\mathcal{A}(v)=g$ in $\Omega$. 

$(a)$ follows from the definition of $\mathcal{M}(\alpha,\beta;\Omega)$, since
\begin{align*}
\langle\mathcal{A}(u)-\mathcal{A}(v),u-v\rangle_{V^*\times V}&=\int_{\Omega}\langle A(x,\nabla_{\mathbb{G}}u)-A(x,\nabla_{\mathbb{G}}v),\nabla_{\mathbb{G}}u-\nabla_{\mathbb{G}}v\rangle\, dx\\
&\geq\alpha\int_{\Omega}|\nabla_\mathbb{G}u-\nabla_\mathbb{G}v|^p\, dx=\alpha\|u-v\|_V^p.
\end{align*}
Moreover, recalling that
\begin{align*}
\left\langle \mathcal{A}(u)-\mathcal{A}(v), u-v\right\rangle_{V^{*}\times V}\leq\|\mathcal{A}(u)-\mathcal{A}(v)\|_{V^*}\|u-v\|_V\quad\forall u,v\in V,
\end{align*}
and applying $(a)$, for $u=\mathcal{A}^{-1}(f)$ and $v=\mathcal{A}^{-1}(g)$, we get
\begin{align*}
\alpha\|\mathcal{A}^{-1}(f)-\mathcal{A}^{-1}(g)\|_V^p\leq\|f-g\|_{V^*}\|\mathcal{A}^{-1}(f)-\mathcal{A}^{-1}(g)\|_V,
\end{align*}
which implies $(b)$.
Finally, $(c)$ holds, since
\begin{align*}
\|A(\cdot,\nabla_{\mathbb{G}}u)-A(\cdot,\nabla_{\mathbb{G}}v)\|_{L^{p'}(\Omega,H\mathbb{G})}\leq\beta\left[|\Omega|+\|u\|_V^p+\|v\|_V^p\right]^\frac{p-2}{p}\|u-v\|_V,
\end{align*}
i.e.,
\begin{align*}
\langle\mathcal{A}(u)-\mathcal{A}(v),u-v\rangle_{V^*\times V}&=\int_{\Omega}\langle A(x,\nabla_{\mathbb{G}}u)-A(x,\nabla_{\mathbb{G}}v),\nabla_{\mathbb{G}}u-\nabla_{\mathbb{G}}v\rangle\, dx\\
&\leq\|A(\cdot,\nabla_{\mathbb{G}}u)-A(\cdot,\nabla_{\mathbb{G}}v)\|_{L^{p'}(\Omega,H\mathbb{G})}\|u-v\|_V\\
&\leq\beta\left[|\Omega|+\|u\|_V^p+\|v\|_V^p\right]^{\frac{p-2}{p}}\|u-v\|_V^2.
\end{align*}
\end{proof}

%%%%%%%%%%%%%%%%%%%%%%%%%%%%%%%%%%%%%%%%%%%

\section{H-convergence and Div-curl lemma}\label{sect4}

Let us recall the notion of \textit{$H$-convergence}. The following statement is the natural adaptation of the original definition of Murat and Tartar to our context.
\begin{defi}\label{Hconv}
Let $(A^n)_n\subset\mathcal{M}(\alpha,\beta;\Omega)$ and let $A^{\text{eff}}\in\mathcal{M}(\alpha',\beta';\Omega)$ for some $\alpha\leq\beta$ and $\alpha'\leq\beta'$ positive constants. We say that $(A^n)_n$ $H$-converges to $A^\text{eff}$ if for every $f\in W^{-1,p'}_{\mathbb{G}}(\Omega)$, given $(u_n)_n\subset W^{1,p}_{\mathbb{G},0}(\Omega)$ the sequence of weak solutions of $-\mathrm{div}_{\mathbb{G}}(A^n(\cdot,\nabla_{\mathbb{G}}u_n))=f$ in $\Omega$ and $u_\infty\in W^{1,p}_{\mathbb{G},0}(\Omega)$ the weak solution of $-\mathrm{div}_{\mathbb{G}}(A^\text{eff}(\cdot,\nabla_{\mathbb{G}}u_\infty))=f$ in $\Omega$, then the following convergences hold:
\begin{itemize}
	\item[$\cdot$] $u_n\rightharpoonup u_\infty$ in $W^{1,p}_{\mathbb{G},0}(\Omega)$-weak;
	\item[$\cdot$] $A^n(\cdot,\nabla_{\mathbb{G}}u_n)\rightharpoonup A^\text{eff}(\cdot,\nabla_{\mathbb{G}}u_\infty)$ in $L^{p'}(\Omega,H\mathbb{G})$-weak.
\end{itemize}
\end{defi}
\medskip

We divide the proof of Theorem \ref{Xth} in several steps. At first, let us show the convergence of the solutions.
\begin{lem}\label{lemma1}
Let $(A^n)_n\subset\mathcal{M}(\alpha,\beta;\Omega)$ and let us consider the sequence of monotone operators $\mathcal{A}_n:W^{1,p}_{\mathbb{G},0}(\Omega)\to W^{-1,p'}_{\mathbb{G}}(\Omega)$, defined as
\begin{align*}
\mathcal{A}_n(u):=-\mathrm{div}_{\mathbb{G}}(A^n(\cdot,\nabla_{\mathbb{G}}u))\ \text{in}\ \Omega.
\end{align*}
Then, there exist a continuous and invertible operator $\mathcal{A}_\infty:W^{1,p}_{\mathbb{G},0}(\Omega)\rightarrow W^{-1,p'}_{\mathbb{G}}(\Omega)$ and a subsequence $(\mathcal{A}_m)_m$ of $(\mathcal{A}_n)_n$, such that
\begin{align*}
    \mathcal{A}_m^{-1}(f)\rightharpoonup\mathcal{A}_\infty^{-1}(f)\quad\text{in }W^{1,p}_{\mathbb{G},0}(\Omega)\text{-weak}
\end{align*}
for every $f\in W^{-1,p'}_{\mathbb{G}}(\Omega)$.
\end{lem}
\begin{proof}
For the sake of simplicity, let us denote by $V=W^{1,p}_{\mathbb{G},0}(\Omega)$ and by $V^*$ be its dual space. We divide the proof of the Lemma in three steps.
\medskip

$[${\it Step 1}$]$ At first, let us show that the sequence of solutions of 
\begin{equation}\label{11}
\mathcal{A}_n(u)=f\ \text{in }\Omega\ \text{for all }n\in\mathbb{N},
\end{equation}
up to subsequences, weakly converges in $V$ for any fixed $f\in X$, where $X$ is a fixed countable and dense subspace of $V^*$. Moreover, we estimate an upper-bound for its limit, in terms of $f$.
\medskip

Let $f\in V^*$ and let $(u_n)_n\subset V$ be the sequence of solutions of (\ref{11}). By Proposition \ref{existence}, $u_n=\mathcal{A}_n^{-1}(f)$ for any $n\in\mathbb{N}$ and, by Proposition \ref{estimates},
\begin{align*}
    \|u_n\|_V\leq\left(\frac{1}{\alpha}\right)^\frac{1}{p-1}\|f\|_{V^*}^\frac{1}{p-1},
\end{align*}
i.e., the sequence $(u_n)_n$ is bounded in V, reflexive Banach space. In terms of diagonal arguments, let $X$ be a countable and dense subset of $V^*$. Then, for any $f\in X$, there exist $u_\infty(f)\in V$ and $(u_m)_m$, diagonal subsequence of $(u_n)_n$, such that
$$u_m\rightharpoonup u_\infty(f)$$
in $V$-weak. Moreover, by the lower semicontinuity of the norm and by Proposition \ref{estimates}, we have
\begin{align*}
\langle f,u_\infty\rangle_{V^*\times V}&=\lim_{m\rightarrow\infty}\langle f,u_m\rangle_{V^*\times V}=\lim_{m\rightarrow\infty}\langle\mathcal{A}_m(u_m),u_m\rangle_{V^*\times V}\\&\geq\alpha\lim_{m\rightarrow\infty}\|u_m\|_V^p\geq\alpha\liminf_{m\rightarrow\infty}\|u_m\|_V^p\geq\alpha\|u_\infty\|_V^p
\end{align*}	
and, since
\begin{align*}
\langle f,u_\infty\rangle_{V^*\times V}\leq\|f\|_{V^*}\|u_\infty\|_V,
\end{align*}
we finally get
\begin{align*}\label{intfy}
\|u_\infty\|_V\leq\left(\frac{1}{\alpha}\right)^\frac{1}{p-1}\|f\|_{V^*}^\frac{1}{p-1}.
\end{align*}
\medskip

$[${\it Step 2}$]$ Defined the operator
\begin{align*}
S:X&\rightarrow V\\f&\mapsto S(f):=\lim_{m\rightarrow\infty}\mathcal{A}_m^{-1}(f),
\end{align*}
let us show that $S$ can be extended to the whole space $V^*$. Since $X$ is a countable and dense subspace of $V^*$, it is enough to show that $S$ is continuous in $(X,\|\cdot\|_{V^*})$.
\medskip

Let $f,g\in X$. Then, by Proposition \ref{estimates}, we have
\begin{align*}
\|\mathcal{A}_m^{-1}(f)-\mathcal{A}_m^{-1}(g)\|_V\leq\left(\frac{1}{\alpha}\right)^\frac{1}{p-1}\|f-g\|_{V^*}^\frac{1}{p-1}\quad\forall m\in\mathbb{N}.
\end{align*}
Therefore, passing to the limit, we get, by the lower semicontinuity of the norm
\begin{align*}
\|S(f)-S(g)\|_V\leq\liminf_{m\rightarrow\infty}\|\mathcal{A}_m^{-1}(f)-\mathcal{A}_m^{-1}(g)\|_V\leq\left(\frac{1}{\alpha}\right)^\frac{1}{p-1}\|f-g\|_{V^*}^\frac{1}{p-1}.
\end{align*}
For the sake of completeness, the extension of $S$ to $V^*\setminus X$ is defined as $$S(f):=\lim_{n\to\infty}S(f_n)$$
where $f\in V^*$, $(f_n)_n\subset X$ are such that $f_n\to f$ in $V^*$.
\medskip

$[${\it Step 3}$]$ Let us finally show the invertibility of $S$ in $V^*$. In order to apply Theorem \ref{thKS}, let us show that $S$ is monotone and coercive in $V^*$.
\medskip

Let $f,g\in V^*$. Thus
\begin{align*}
\left\langle S(f)-S(g),f-g\right\rangle_{V\times V^{*}}&=\lim_{m\rightarrow\infty}\left\langle \mathcal{A}_m^{-1}(f)-\mathcal{A}_m^{-1}(g),f-g\right\rangle_{V\times V^{*}}\\&=\lim_{m\rightarrow\infty}\left\langle \mathcal{A}_m(u_m)-\mathcal{A}_m(v_m),u_m-v_m\right\rangle_{V^*\times V}\\&\geq\alpha\lim_{m\rightarrow\infty}\|u_m-v_m\|_V^p\geq 0.
\end{align*}
Moreover, by Proposition \ref{estimates}
\begin{align*}
\|\mathcal{A}_m(u_m)-\mathcal{A}_m(v_m)\|_{V^*}^p
&\leq\beta^p\left[|\Omega|+\|u_m\|_V^p+\|v_m\|_V^p\right]^{p-2}\|u_m-v_m\|_V^p\\
&\leq\frac{\beta^p}{\alpha}\left[|\Omega|+\|u_m\|_V^p+\|v_m\|_V^p\right]^{p-2}\\
&\ \ \ \ \langle\mathcal{A}_m(u_m)-\mathcal{A}_m(v_m),u_m-v_m\rangle_{V^*\times V}\\
&\leq\frac{\beta^p}{\alpha}\left[|\Omega|+\left(\frac{1}{\alpha}\right)^{p'}\|f\|_{V^*}^{p'}+\left(\frac{1}{\alpha}\right)^{p'}\|g\|_{V^*}^{p'}\right]^{p-2}\\
&\ \ \ \ \langle\mathcal{A}^{-1}_m(f)-\mathcal{A}^{-1}_m(g),f-g\rangle_{V\times V^*}
\end{align*}
and, passing to the limit
\begin{align*}
\|f-g\|_{V^*}^p&\leq\frac{\beta^p}{\alpha}\left[|\Omega|+\left(\frac{1}{\alpha}\right)^{p'}\|f\|_{V^*}^{p'}+\left(\frac{1}{\alpha}\right)^{p'}\|g\|_{V^*}^{p'}\right]^{p-2}\\&\ \ \ \ \langle S(f)-S(g),f-g\rangle_{V\times V^*}.
\end{align*}
We get the thesis, defining
$$\mathcal{A}_\infty:=S^{-1}:V\rightarrow V^*.$$
\end{proof}
We continue by proving the convergence of the momenta.
\begin{lem}\label{lemma2}
Let $(\mathcal{A}_n)_n$ be as in the previous Lemma. Then, for every $f\in W^{-1,p'}_{\mathbb{G}}(\Omega)$, there exists a continuous operator $M:W^{-1,p'}_{\mathbb{G}}(\Omega)\rightarrow L^{p'}(\Omega,H\mathbb{G})$ such that, up to subsequences, $A^n(\cdot,\nabla_\mathbb{G}\mathcal{A}_n^{-1}(f))\rightharpoonup M(f)$ in $L^{p'}(\Omega,H\mathbb{G})$-weak.
\end{lem}
\begin{proof}
In terms of diagonal arguments, let $X$ be a countable and dense subspace of $L^{p'}(\Omega,H\mathbb{G})$ and let $f\in X$. Repeating the same techniques of the previous lemma, let us show the existence of a diagonal subsequence of $(A^n(\cdot,\nabla_{\mathbb{G}}\mathcal{A}_n^{-1}(f)))_n$, weakly convergent in $L^{p'}(\Omega,H\mathbb{G})$.
\medskip

Since $(A^n)_n\subset\mathcal{M}(\alpha,\beta;\Omega)$, by the H\"older inequality we get
\begin{align*}
\int_{\Omega}|A^n(x,\nabla_{\mathbb{G}}\mathcal{A}_n^{-1}(f))|^{p'}\, dx&\leq\beta^{p'}\int_{\Omega}\left[1+|\nabla_{\mathbb{G}}\mathcal{A}_n^{-1}(f)|^p\right]^\frac{p-2}{p-1}|\nabla_{\mathbb{G}}\mathcal{A}_n^{-1}(f)|^{p'}\, dx\\
&\leq\beta^{p'}\left(\int_{\Omega}\left[1+|\nabla_{\mathbb{G}}\mathcal{A}_n^{-1}(f)|^p\right]\, dx\right)^\frac{p-2}{p-1}\\
&\quad\left(\int_{\Omega}|\nabla_{\mathbb{G}}\mathcal{A}_n^{-1}(f)|^{p}\, dx\right)^\frac{p'}{p}\\
&=\beta^{p'}\left[|\Omega|+\|\mathcal{A}_n^{-1}(f)\|_V^p\right]^{\frac{p-2}{p}p'}\|\mathcal{A}_n^{-1}(f)\|_V^{p'},
\end{align*}
i.e.,
\begin{align*}
\|A^n(\cdot,\nabla_{\mathbb{G}}\mathcal{A}_n^{-1}(f))\|_{L^{p'}(\Omega,H\mathbb{G})}\leq\beta\left[|\Omega|+\|\mathcal{A}_n^{-1}(f)\|_V^p\right]^\frac{p-2}{p}\|\mathcal{A}_n^{-1}(f)\|_V.
\end{align*}
Therefore, by Proposition \ref{estimates}, we have
\begin{align*}
\|A^n(\cdot,\nabla_{\mathbb{G}}\mathcal{A}_n^{-1}(f))\|_{L^{p'}(\Omega,H\mathbb{G})}\leq\frac{\beta}{\alpha^\frac{1}{p-1}}\left[|\Omega|+\left(\frac{1}{\alpha}\right)^{p'}\|f\|_{V^*}^{p'}\right]^\frac{p-2}{p}\|f\|_{V^*}^\frac{1}{p-1},
\end{align*}
that is, the sequence $(A^n(\cdot,\nabla_{\mathbb{G}}\mathcal{A}_n^{-1}(f)))_n$ is bounded in $L^{p'}(\Omega,H\mathbb{G})$.

By the countability of $X$, there exists a diagonal subsequence of $(A^n(\cdot,\nabla_{\mathbb{G}}\mathcal{A}_n^{-1}(f)))_n$, weakly convergent to $M=M(f)$ in $L^{p'}(\Omega,H\mathbb{G})$.

Let us now define the operator
\begin{align*}
M:X&\rightarrow L^{p'}(\Omega,H\mathbb{G})\\f&\mapsto M(f):=\lim_{m\rightarrow\infty}A^m(\cdot,\nabla_{\mathbb{G}}\mathcal{A}_m^{-1}(f)).
\end{align*}
In order to conclude the proof, we need to extend $M$ to the whole space $V^*$. We just need a continuous condition on $M$.

Let $f,g\in X$. By Proposition \ref{estimates}, we have
\begin{align*}
\|A^m(\cdot,&\nabla_{\mathbb{G}}\mathcal{A}_m^{-1}(f))-A^m(\cdot,\nabla_{\mathbb{G}}\mathcal{A}_m^{-1}(g))\|_{L^{p'}(\Omega,H\mathbb{G})}\\&\leq\frac{\beta}{\alpha^\frac{1}{p-1}}\left[|\Omega|+\left(\frac{1}{\alpha}\right)^{p'}\|f\|_{V^*}^{p'}+\left(\frac{1}{\alpha}\right)^{p'}\|g\|_{V^*}^{p'}\right]^\frac{p-2}{p}\|f-g\|_{V^*}^\frac{1}{p-1}.
\end{align*}
Then, by the lower semicontinuity of the norm, we finally have
\begin{align*}
\|M(f)&-M(g)\|_{L^{p'}(\Omega,H\mathbb{G})}\\&\leq\frac{\beta}{\alpha^\frac{1}{p-1}}\left[|\Omega|+\left(\frac{1}{\alpha}\right)^{p'}\|f\|_{V^*}^{p'}+\left(\frac{1}{\alpha}\right)^{p'}\|g\|_{V^*}^{p'}\right]^\frac{p-2}{p}\|f-g\|_{V^*}^\frac{1}{p-1},
\end{align*}
which concludes the proof.
\end{proof}
\medskip

In order to conclude the proof of Theorem \ref{Xth}, we need to recall a classic tool: the {\it Div-curl Lemma}. The following statement, in the context of Carnot groups, was given by Baldi, Franchi, Tchou and Tesi in \cite{BFTT}. The definition of \textit{intrinsic curl}, $\mathrm{curl}_{\mathbb{G}}$, can be found in \cite[Section 5]{BFTT}.
\begin{teo}\cite[Theorem 5.1]{BFTT}\label{divcurl}
	Let $\Omega\subset\mathbb{G}$ be an open set and let $p,q>1$ a H\"older conjugate pair. Moreover, following the notations of \cite{BFTT}, if $\sigma\in\mathcal{I}^2_0$, let $a(\sigma)>1$ and $b>1$ be such that
$$a(\sigma)>\frac{Qp}{Q+(\sigma-1)p}\ \text{and}\ b>\frac{Qq}{Q+q}.$$
Let now $(E^n)_n\subset L^p_{loc}(\Omega,H\mathbb{G})$ and $(D^n)_n\subset L^q_{loc}(\Omega,H\mathbb{G})$ be sequences of horizontal vector fields for $n\in\mathbb{N}$ weakly convergent to $E\in L^p_{loc}(\Omega,H\mathbb{G})$ and $D\in L^q_{loc}(\Omega,H\mathbb{G})$, respectively, and such that:
\begin{itemize}
	\item[$\cdot$] the components of $(\mathrm{curl}_{\mathbb{G}}E^n)_n$ of weight $\sigma$ are bounded in $L^{a(\sigma)}_{loc}(\Omega,H\mathbb{G})$;
	\item[$\cdot$] $(\mathrm{div}_\mathbb{G}D^n)_n$ is bounded in $L^b_{loc}(\Omega,H\mathbb{G})$.
\end{itemize}
Then
\begin{align*}
\langle D^n,E^n\rangle\to\langle D,E\rangle\ \text{in}\ \mathcal{D}'(\Omega),
\end{align*}
i.e.,
\begin{align*}
\int_{\Omega}\langle D^n(x),E^n(x)\rangle\varphi(x)\, dx\to\int_{\Omega}\langle D(x),E(x)\rangle\varphi(x)\, dx
\end{align*}
for any $\varphi\in\mathcal{D}(\Omega)$.
\end{teo}
\begin{proof}[\bf Proof of Theorem~\ref{Xth}]
Given the operators $\mathcal{A}_\infty$ and $M$, defined in the proofs of Lemma \ref{lemma1} and Lemma \ref{lemma2}, let us consider the following composition
$$C:=M\circ\mathcal{A}_\infty:W^{1,p}_{\mathbb{G},0}(\Omega)\rightarrow L^{p'}(\Omega,H\mathbb{G}),$$
and let us prove the existence of $A^{\text{eff}}\in\mathcal{M}(\alpha,\beta;\Omega)$, such that
\begin{align*}
C(u)=A^{\text{eff}}(x,\nabla_\mathbb{G}\mathcal{A}_\infty^{-1}(f))
\end{align*}
for every $f\in W^{-1,p'}_\mathbb{G}(\Omega)$ and for all $u\in W^{1,p}_{\mathbb{G},0}(\Omega)$ such that $\mathcal{A}_\infty(u)=f$ a.e. $x\in\Omega$.
\medskip

For any fixed $f\in W^{-1,p'}_\mathbb{G}(\Omega)$, given $\omega$ an open set such that $\overline{\omega}\subset\Omega$ and given $v\in W^{1,p}_{\mathbb{G},0}(\Omega)$, a solution of $-\mathrm{div}_{\mathbb{G}}(A(\cdot,\nabla_{\mathbb{G}}v))=f$ in $\Omega$, we define the Carath\'eodory function $A^\text{eff}: \omega\times\mathbb{R}^{m}\to \mathbb{R}^{m}$ as
$$A^\text{eff}(x,\xi):=C(v),$$
if $\nabla_\mathbb{G}v(x)=\xi$ a.e. $x$ in $\omega$.
\medskip

In order to show that this definition makes sense, we need to verify that $A^\text{eff}(x,\xi_1)=A^\text{eff}(x,\xi_2)$ a.e. $x$ in $\omega_1\cap\omega_2$, for any $\xi_1,\xi_2\in\mathbb{R}^{m}$ such that $\xi_1=\xi_2$ and for any $\omega_1,\omega_2$ open sets such that $\overline{\omega_1},\overline{\omega_2}\subset\Omega$.

For $i=1,2$, let us consider the following spaces, functions and sequences:
\begin{itemize}
	\item[$(a)$] $\omega_i$ open sets such that $\overline{\omega_i}\subset\Omega$;
	\item[$(b)$] $\varphi_i\in C^1_c(\Omega)$ such that $\varphi_i|_{\omega_i}=1$;
	\item[$(c)$] $\xi_i\in\mathbb{R}^{m}$;
	\item[$(d)$] $(v_{i,n})_n\subset W^{1,p}_{\mathbb{G},0}(\Omega)$ weakly convergent, up to subsequences, in $W^{1,p}_{\mathbb{G},0}(\Omega)$ to $$v_{i,\infty}(x)=\varphi_i(x)\langle\xi_i,\pi(x)\rangle$$
	where $\pi(x)=(x_1,..,x_{m})$ for every $x=(x_1,..,x_{n})\in\Omega$;
	\item[$(e)$] $f_i\in W^{-1,p'}_\mathbb{G}(\Omega)$ such that $f_i=-\mathrm{div}_\mathbb{G}C(v_{i,\infty})$;
	\item[$(f)$] $(D_i^n)_n:=(A^n(\cdot,\nabla_{\mathbb{G}}v_{i,n}))_n\subset L^{p'}(\Omega,H\mathbb{G})$;
	\item[$(g)$] $(E_i^n)_n:=(\nabla_{\mathbb{G}}v_{i,n})_n\subset L^p(\Omega,H\mathbb{G})$.
\end{itemize}

By definition, $\nabla_\mathbb{G}v_{i,\infty}=\xi_i$ in $\omega_i$, for $i=1,2$. Moreover, in terms of diagonal arguments, there exist $(D_i^m)_m$ and $(E_i^m)_m$, diagonal subsequences of $(D_i^n)_n$ and $(E_i^n)_n$, respectively, and there exist $D_i\in L^{p'}(\Omega,H\mathbb{G})$ and $E_i\in L^p(\Omega,H\mathbb{G})$ such that:
\begin{itemize}
	\item[$\cdot$] $D_i^m\rightharpoonup D_i$ in $L^{p'}(\Omega,H\mathbb{G})$-weak
	\item[$\cdot$] $E_i^m\rightharpoonup E_i$ in $L^p(\Omega,H\mathbb{G})$-weak.
\end{itemize}

Since $\mathrm{curl}_\mathbb{G}(E_i^m)=0$
for all $m$, $D_i(\cdot)=A^{\text{eff}}(\cdot,\xi_i)$ and $E_i=\xi_i$ in $\omega_i$ for $i=1,2$, the hypotheses of Theorem \ref{divcurl} are satisfied (taking as $a$ each value grater than $1$ that satisfies the hypotheses of the Lemma, and $b=p'$). Therefore
\begin{align*}
\int_{\Omega}\langle D_2^m-D_1^m,E_2^m-E_1^m\rangle\varphi\, dx\to\int_{\Omega}\langle D_2-D_1,E_2-E_1\rangle\varphi\, dx
\end{align*}
for any $\varphi\in\mathcal{D}(\omega_1\cap\omega_2)$, i.e.,
\begin{align*}
\int_{\Omega}\langle A^m(x,\nabla_{\mathbb{G}}v_{2,m})-A^m(x,\nabla_{\mathbb{G}}v_{1,m}),\nabla_{\mathbb{G}}v_{2,m}-\nabla_{\mathbb{G}}v_{1,m}\rangle\varphi(x)\, dx\\\to\int_{\Omega}\langle A^{\text{eff}}(x,\xi_2)-A^{\text{eff}}(x,\xi_1),\xi_2-\xi_1\rangle\varphi(x)\, dx
\end{align*}
for any $\varphi\in\mathcal{D}(\omega_1\cap\omega_2)$.

Let us assume $\varphi\geq 0$. Since $(A^m)_m\subset\mathcal{M}(\alpha,\beta;\Omega)$, then
\begin{align*}
&\int_{\Omega}\langle A^{\text{eff}}(x,\xi_2)-A^{\text{eff}}(x,\xi_1),\xi_2-\xi_1\rangle\varphi(x)\, dx\\
&\geq\liminf_{m\rightarrow\infty}\int_{\Omega}\langle A^m(x,\nabla_{\mathbb{G}}v_{2,m})-A^m(x,\nabla_{\mathbb{G}}v_{1,m}),\nabla_{\mathbb{G}}v_{2,m}-\nabla_{\mathbb{G}}v_{1,m}\rangle\varphi(x)\, dx\\
&\geq\liminf_{m\rightarrow\infty}\alpha\int_{\Omega}|\nabla_\mathbb{G}v_{2,m}-\nabla_\mathbb{G}v_{1,m}|^p\varphi(x)\, dx\\
&\geq\alpha\int_{\Omega}|\nabla_\mathbb{G}v_{2,\infty}-\nabla_\mathbb{G}v_{1,\infty}|^p\varphi(x)\, dx\\
&=\alpha\int_{\Omega}|\xi_2-\xi_1|^p\varphi(x)\, dx
\end{align*}
and
\begin{align*}
&\int_{\Omega}\langle A^{\text{eff}}(x,\xi_2)-A^{\text{eff}}(x,\xi_1),\xi_2-\xi_1\rangle\varphi(x)\, dx\\
&\geq\liminf_{m\rightarrow\infty}\alpha\int_{\Omega}|\nabla_\mathbb{G}v_{2,m}-\nabla_\mathbb{G}v_{1,m}|^p\varphi(x)\, dx\\
&\geq\liminf_{m\rightarrow\infty}\dfrac{\alpha}{\beta^p}\int_{\Omega}\left[1+|\nabla_\mathbb{G}v_{2,m}|^p+|\nabla_\mathbb{G}v_{1,m}|^p\right]^{2-p}|A^m(x,\nabla_\mathbb{G}v_{2,m})-A^m(x,\nabla_\mathbb{G}v_{1,m})|^p\varphi(x)\, dx\\
&\geq\dfrac{\alpha}{\beta^p}\int_{\Omega}\left[1+|\nabla_\mathbb{G}v_{2,\infty}|^p+|\nabla_\mathbb{G}v_{1,\infty}|^p\right]^{2-p}|A^\text{eff}(x,\nabla_\mathbb{G}v_{2,\infty})-A^\text{eff}(x,\nabla_\mathbb{G}v_{1,\infty})|^p\varphi(x)\, dx\\
&=\dfrac{\alpha}{\beta^p}\int_{\Omega}\left[1+|\xi_2|^p+|\xi_1|^p\right]^{2-p}|A^\text{eff}(x,\xi_2)-A^\text{eff}(x,\xi_1)|^p\varphi(x)\, dx.
\end{align*}
Restricting the integrals to $\omega_1\cap\omega_2$ and varying $\varphi$, we get
\begin{itemize}
	\item[$(a)$] $\langle A^{\text{eff}}(x,\xi_2)-A^{\text{eff}}(x,\xi_1),\xi_2-\xi_1\rangle\geq\alpha|\xi_2-\xi_1|^p$
	\item[$(b)$] $\langle A^{\text{eff}}(x,\xi_2)-A^{\text{eff}}(x,\xi_1),\xi_2-\xi_1\rangle\geq\dfrac{\alpha}{\beta^p}\left[1+|\xi_2|^p+|\xi_1|^p\right]^{2-p}\\|A^\text{eff}(x,\xi_2)-A^\text{eff}(x,\xi_1)|^p$
\end{itemize}
a.e. $x\in\omega_1\cap\omega_2$.

For $\xi_1=\xi_2$, we obtain $A^\text{eff}(x,\xi_1)=A^\text{eff}(x,\xi_2)$ a.e. $x\in\omega_1\cap\omega_2$, while, taking $\xi_1\neq\xi_2$, and recalling that $A^\text{eff}(\cdot,0)=0$ by definition, we deduce that
$$A^\text{eff}\in\mathcal{M}(\alpha,\beta;\Omega).$$
\begin{oss}\label{rem}
In order to prove condition $(iii)$ of Definition \ref{M,alpha,beta}, it is enough to observe that
\begin{align*}
&\int_{\Omega}|\xi_2-\xi_1|^p\varphi(x)\, dx\\
&\geq\liminf_{m\rightarrow\infty}\int_{\Omega}|\nabla_\mathbb{G}v_{2,m}-\nabla_\mathbb{G}v_{1,m}|^p\varphi(x)\, dx\\
&\geq\liminf_{m\rightarrow\infty}\dfrac{1}{\beta^p}\int_{\Omega}\left[1+|\nabla_\mathbb{G}v_{2,m}|^p+|\nabla_\mathbb{G}v_{1,m}|^p\right]^{2-p}|A^m(x,\nabla_\mathbb{G}v_{2,m})-A^m(x,\nabla_\mathbb{G}v_{1,m})|^p\varphi(x)\, dx\\
&\geq\dfrac{1}{\beta^p}\int_{\Omega}\left[1+|\nabla_\mathbb{G}v_{2,\infty}|^p+|\nabla_\mathbb{G}v_{1,\infty}|^p\right]^{2-p}|A^\text{eff}(x,\nabla_\mathbb{G}v_{2,\infty})-A^\text{eff}(x,\nabla_\mathbb{G}v_{1,\infty})|^p\varphi(x)\, dx\\
&=\dfrac{1}{\beta^p}\int_{\Omega}\left[1+|\xi_2|^p+|\xi_1|^p\right]^{2-p}|A^\text{eff}(x,\xi_2)-A^\text{eff}(x,\xi_1)|^p\varphi(x)\, dx.
\end{align*}
\end{oss}
\medskip

Let now $u_\infty\in W^{1,p}_{\mathbb{G},0}(\Omega)$ be the (unique) solution of $\mathcal{A}_\infty(u)=f$ in $\Omega$ and let $(u_m)_m\subset W^{1,p}_{\mathbb{G},0}(\Omega)$ be weakly convergent to $u_\infty$ in $W^{1,p}_{\mathbb{G},0}(\Omega)$. In order to conclude the proof of the theorem, we have to show that
$$C(u_\infty)=A^\text{eff}(x,\nabla_\mathbb{G}u_\infty)\quad\text{a.e. } x\in\Omega.$$
To achieve this purpose, we apply one more time Theorem \ref{divcurl}, for $D_2^m=A^m(x,\nabla_\mathbb{G}u_m)$ and $E_2^m=\nabla_\mathbb{G}u_m$.
\medskip

For every $\varphi\in\mathcal{D}(\omega_1)$ such that $\varphi\geq 0$, we have
\begin{align*}
\int_{\Omega}\langle A^m(x,\nabla_{\mathbb{G}}u_m)-A^m(x,\nabla_{\mathbb{G}}v_{1,m}),\nabla_{\mathbb{G}}u_m-\nabla_{\mathbb{G}}v_{1,m}\rangle\varphi(x)\, dx\\\to\int_{\Omega}\langle C(u_\infty)-A^{\text{eff}}(x,\xi_1),\nabla_\mathbb{G}u_\infty-\xi_1\rangle\varphi(x)\, dx.
\end{align*}
Following the techniques of the first part of the proof, we get
\begin{itemize}
	\item[$(a)$] $\langle C(u_\infty)-A^{\text{eff}}(x,\xi_1),\nabla_\mathbb{G}u_\infty-\xi_1\rangle\geq\alpha|\nabla_\mathbb{G}u_\infty-\xi_1|^p$
	\item[$(b)$] $\langle C(u_\infty)-A^{\text{eff}}(x,\xi_1),\nabla_\mathbb{G}u_\infty-\xi_1\rangle\geq\dfrac{\alpha}{\beta^p}\left[1+|\nabla_\mathbb{G}u_\infty|^p+|\xi_1|^p\right]^{2-p}\\|C(u_\infty)-A^\text{eff}(x,\xi_1)|^p$
\end{itemize}
and, in particular, that
\begin{align*}
\left|C(u_\infty)-A^\text{eff}(x,\xi_1)\right|\leq\beta\left[1+|\nabla_\mathbb{G}u_\infty|^p+|\xi_1|^p\right]^\frac{p-2}{p}\left|\nabla_\mathbb{G}u_\infty-\xi_1\right|
\end{align*}
a.e. $x\in\omega_1$ (see Remark \ref{rem} for details). Varying $\omega_1$ and $\xi_1\in\mathbb{R}^{m}$, we get the thesis.
\end{proof}

\section*{Acknowledgements}
The author would like to thank Prof. Francesco Serra Cassano and Prof. Andrea Pinamonti, of the Department of Mathematics of the University of Trento, for their support and help.

%%%%%%%%%%%%%%%%%%%%%%%%%%%%%%%%%

\end{document}